\documentclass[11pt]{amsart}
\usepackage{amsmath}
\usepackage{amsfonts}
\usepackage[mathscr]{eucal}
\usepackage{amssymb}
\usepackage{marginnote}
\usepackage{geometry,amsthm,amsmath,amssymb, graphicx,  float, enumerate}
\usepackage{color}
\usepackage[all]{xy}
\newtheorem{thm}{Theorem}[section]
\newtheorem{cor}[thm]{Corollary}

\newtheorem{lemma}[thm]{Lemma}
\newtheorem{prop}[thm]{Proposition}

\theoremstyle{definition}
\newtheorem{defn}[thm]{Definition}
\newtheorem{remark}[thm]{Remark}

\newcommand{\Span}{\operatorname{span}}
\newcommand{\bb}[1]{\mathbb{#1}}
\newcommand{\cl}[1]{\mathcal{#1}}

\newcommand{\Ran}{\operatorname{Ran}}

\newcommand{\Inner}[2]{\left\langle {#1}, {#2} \right\rangle}
\newcommand{\inner}[2]{\langle {#1}, {#2} \rangle}
\newcommand{\Ker}{\operatorname{Ker}}

\newcommand{\veq}{\mathrel{\rotatebox{90}{$=$}}}

\numberwithin{equation}{section}
 \allowdisplaybreaks[3]
\begin{document}

%\setcounter{thm}{1}
% \title[short text for running head]{full title}
\title[Tensor Products of Cuntz Operator Systems]{Tensor
  Products of the Operator System Generated by the Cuntz Isometries}

%    Only \author and \address are required; other information is
%    optional.  Remove any unused author tags.

%    author one information
% \author[short version for running head]{name for top of paper}
\author{V.I. Paulsen} 
\address{Department of Mathematics, University of Houston; Department of Pure Mathematics and Institute for Quantum Computing, University of Waterloo}
\email{vpaulsen@uwaterloo.ca}
\author{Da Zheng}
\address{Department of Mathematics, University of Houston}
\email{dazheng@math.uh.edu}
\thanks{This research was supported in part by NSF grant DMS-1101231.}
\date{\today}
\begin{abstract}
We study tensor products and nuclearity-related properties of the operator system $\cl S_n$ generated by the Cuntz isometries. By using the nuclearity of the Cuntz algebra,
we can show that $\cl{S}_n$ is $C^*$-nuclear, and this implies a dual row contraction version of Ando's theorem characterizing operators of numerical radius 1. On the other hand, without using the nuclearity of the Cuntz algebra, we are still able to show directly this Ando type property of dual row contractions and conclude 
that $\cl{S}_n$ is $C^*$-nuclear,  which yields a new proof of the nuclearity of the Cuntz algebras.
 We prove that the dual operator system of $\cl{S}_n$ is completely order isomorphic to an operator subsystem of $M_{n+1}$.
 %we observe that a map $\phi:\cl{S}_n^d\to \cl{A}$
%being unitally completely positive is equivalent to the images of the generators being a co-row contraction, where $\cl{A}$ is a $C^*$-algebra.
Finally, a lifting result concerning Popescu's joint numerical radius is proved via operator system techniques.
\end{abstract}
\maketitle
%%%%%%%%%%%%%%%%%%%%%%%%%%%%%%%%%%%%%%%%%%%%%%%%%%%%%%%%%
 \section{Introduction}
 The operator system generated by Cuntz isometries is studied in \cite{ZHE14}. We let $S_1,\dots,S_n$ be $n$ ($2\leq n<+\infty$) generators of the Cuntz algebra $\cl{O}_n$ and $I$ be the identity operator,  and let $\cl{S}_n$ denote the operator system generated by $S_1,\dots,S_n$, that is, $\cl{S}_n=\Span\{I,S_1,\dots,S_n,S_1^*,\dots,S^*_n\}$. Similarly, we let $S_1,\dots,S_n,\dots$ be the generators of $\cl{O}_\infty$ and set $\cl{S}_\infty=\Span\{I,S_1,\dots,S_n,\dots,S_1^*,\dots,S^*_n,\dots\}$. 

In this paper, we turn our attention to tensor products and nuclearity-related properties of $\cl{S}_n$, which is motivated by the well-known fact that $\cl{O}_n$ ($2\leq n\leq \infty$) is nuclear in the sense that for every unital $C^*$-algebra $\cl{A}$,
\[ \cl{O}_n\otimes_{\min} \cl{A}=\cl{O}_n\otimes_{\max} \cl{A} .\]
Since $\cl{S}_n$ contains all the generators of $\cl{O}_n$ and its $C^*$-envelope coincides with $\cl{O}_n$, it is natural to study tensor properties of $\cl{S}_n$ ($2\leq n\leq \infty$)  in the operator system category. 

%(min, max)-nuclear. However, using the nuclearity of the Cuntz algebra, we are able to %show that $\cl{S}_n$ is $C^*$-nuclear. And this fact leads to some useful results. 
%The foundation of tensor products of operator systems is systematically studied in %\cite{KPTT11}. 
Of course, we hope that $\cl{S}_n$ is nuclear in the operator system category. Unfortunately,
according to Definition \ref{defminmaxnuclear}, we can show  that $\cl{S}_n$ is not (min, max)-nuclear by constructing a counter-example. 
However, (min, max)-nuclearity is quite a strong condition for an operator system, as it has been shown  that a finite dimensional operator system is (min, max)-nuclear if and only if it is completely order isomorphic to a $C^*$-algebra if and only if it is injective  \cite[Theorem 6.11]{KPTT13}. So we make a concession and ask whether $\cl{S}_n$ is $C^*$-nuclear (see Definition \ref{defcnulcear}). Fortunately, the answer is affirmative for this case. This fact follows from a refined version of Bunce's dilation theorem for row-contractions \cite[Proposition 1]{BUN84} and the fact that $\cl O_n$ is nuclear. Thus, the operator system $\cl{S}_n$ enjoys many nice properties such as WEP, OSLLP, DCEP, exactness, etc (See \cite{KPTT13}).

On the other hand, it is tempting to show directly that $\cl{S}_n$ is $C^*$-nuclear, that is, without using the nuclearity of $\cl{O}_n$. This is motivated by our result that $\cl{O}_n$ is nuclear if and only if $\cl{S}_n$ is $C^*$-nuclear.   We are able to show the latter directly by using operator system techniques together with the theory of shorted operators. This provides us with a new  proof  of the nuclearity of the Cuntz algebras. Moreover, it motivates us to approach some important properties of the Cuntz algebra via operator system techniques. 
%In addition, as Cuntz algebra is a special case of the $C^*$-algebras generated by a family of operators with Cuntz-Krieger relation of directed graphs, one may hope to study more general operator systems whose generators satisfy the Cuntz-Krieger relation.
This direct proof of the $C^*$-nuclearity of $\cl{S}_n$ also yields a dual row contraction version of Ando's theorem characterizing operators of numerical radius 1.

Recently, Kavruk \cite{KAV14b} showed that for a finite dimensional operator system, $C^*$-nuclearity passes to its dual operator system, and vice versa. This motivates us to study the dual operator system of $\cl{S}_n$, which we denote by $\cl{S}_n^d$.  We show that $\cl{S}_n^d$ is completely order isomorphic to an operator subsystem of $M_{n+1}$. By Kavruk's result, we know that this operator system is also C*-nuclear. However, we were unable to give a direct proof that this operator subsystem is $C^*$-nuclear, although an operator system in the matrix algebras seems easier to deal with.   
From the general theory of operator system tensor products, we know that $C^*$-nuclearity is stronger than a lifting property, the OSLLP. Since $\cl{S}_n^d$ is $C^*$-nuclear, it has the OSLLP and we use this fact to prove a lifting property for Popescu's joint numerical radius for $n$-tuples of operators.
 
 %might be easier to approach. Also, the result that 
%$\cl{S}_n$ is $C^*$-nuclear leads to some other useful results which we shall see in the following discussion. 

In section 2, we will first show that $\cl{S}_n$ is not (min, max)-nuclear by giving a counter-example. Next, we prove that $O_n \otimes_{\min} \cl A = \cl O_n \otimes_{\max} \cl A$ for a unital $C^*$-algebra $\cl{A}$ if and only if $S_n \otimes_{\min} \cl A = \cl S_n \otimes_{\max} \cl A$. Because $\cl{O}_n$ is nuclear, it follows that $\cl{S}_n$ is $C^*$-nuclear.

In section 3,  we will show that the operator system $\cl{E}_n$ defined in \cite{ZHE14} is
$C^*$-nuclear. We then use this fact together with some operator system methods to prove that for a
unital C*-algebra $\cl A$, the equality $\cl
O_n \otimes_{\min} \cl A = \cl O_n \otimes_{\max} \cl A$
is equivalent to a lifting property that must be met by $\cl A.$
Thus, using the fact that $\cl O_n$ is nuclear, we have that every
C*-algebra enjoys this lifting property. A direct corollary of this result shows a contraction 
version of Ando's theorem. 

%Alternatively, by proving
%directly that every C*-algebra enjoys this property, one can give a
%new proof of the nuclearity of $\cl O_n.$ 
In section 4, we will give a direct proof of $S_n \otimes_{\min} \cl A = \cl S_n \otimes_{\max} \cl A$ for a unital $C^*$-algebra $\cl{A}$,  without using the nuclearity of $\cl{O}_n$, by showing that every unital $C^*$-algebra does enjoy the lifting property mentioned in section 3. So we obtain a new proof of the nuclearity of the Cuntz algebra.

In section 5, we will study the dual operator system of $\cl{S}_n$, say $\cl{S}_n^d$. 
We will give characterizations for positive elements in $M_p(\cl{S}_n^d)$ and a necessary and sufficient 
condition for unital completely positive maps from $\cl{S}_n^d$ into a $C^*$-algebras. We will also show that $\cl{S}_n^d$ is not (min, max)-nulcear. By a result of Kavruk (Theorem \ref{nucleartodual}), , $\cl{S}_n^d$ is  $C^*$-nulcear. Moreover, it is interesting to know that $\cl{S}_n^d$ is completely order isomorphic to an operator subsystem in $M_{n+1}$, in contrast to the fact
that $\cl{S}_n$ is completely order isomorphic to the quotient of an operator subsystem in $M_{n+1}$\cite[Theorem 5.1]{ZHE14}.

In section 6, We observe that unital completely positive maps from $\cl{S}_n^d$ to $C^*$-algebras are closely related to the joint numerical radius of $n$-tuples. Then, a lifting property of the joint numerical radius is proved by using the fact that $\cl{S}_n^d$ has the lifting property. 

We shall assume that the reader is familiar with the basics of operator system tensor products (see \cite{KPTT11} or \cite{KAV14}, e.g.). Also, we refer the reader to \cite{KPTT13} for the basics of operator system quotients.

For the rest of this introductory section, we briefly introduce some terminologies and results from the theory of operator system tensor products which will be used throughout this paper.

\begin{defn}
Let $\cl{S}$ and $\cl{T}$ be operators systems. A map $\phi:\cl{S}\to \cl{T}$ is called a  \textbf{complete order isomorphism} if $\phi$ is a unital  linear
isomorphism and both $\phi$ and $\phi^{-1}$ are completely positive, and we say that $\cl{S}$ is completely order isomorphic to $\cl{T}$ if such $\phi$
exists. 
A map $\phi$ is called a \textbf{complete order injection or embedding} if it is a complete order isomorphism onto its range with $\phi(1_{\cl{S}})$ being an
Archimedean order unit. We shall denote this by $\cl{S}\subseteq_{\operatorname{c.o.i}}\cl{T}$.
\end{defn}
Given operator systems $\cl S$ and $\cl T$ and two possibly different
operator system structures $\cl S \otimes_{\alpha} \cl T$ and $\cl S
\otimes_{\beta} \cl T$ on their tensor product, we shall write $\cl S
\otimes_{\alpha} \cl T = \cl S \otimes_{\beta} \cl T$ to mean that the
identity map is a complete order isomorphism.

The tensor products of operator systems we will use in this paper are: min, max, c (See \cite{KPTT11} for their definitions).
The relationship between these tensor products is $\min \leq c\leq \max$, that is, the identity maps $\operatorname{id}:\cl{S}\otimes_{\max}\cl{T} \to\cl{S}\otimes_c\cl{T}$,  $\operatorname{id}:\cl{S}\otimes_c\cl{T} \to\cl{S}\otimes_{\min}\cl{T}$ are completely positive.

\begin{defn}\cite{KPTT11}\label{defminmaxnuclear}
An operator system $\cl{S}$ is called {\bf (min, max)-nuclear} if
$\cl{S}\otimes_{\min}\cl{T}=\cl{S}\otimes_{\max}\cl{T},$ for every
operator system $\cl{T}$.
\end{defn}

\begin{defn}\cite{KPTT11}\label{defcnulcear}
An operator system $\cl{S}$ is called {\bf $C^*$-nuclear} if $\cl{S}\otimes_{\min}\cl{A}=\cl{S}\otimes_{\max}\cl{A}$ for every
unital $C^*$-algebra $\cl{A}$.
\end{defn}
\begin{prop}\cite{KAV14}
An operator system $\cl S$ is $C^*$-nuclear if and
only if $\cl S \otimes_{\min} \cl T = \cl S \otimes_c \cl T$ for every
operator system $\cl T$.
\end{prop}

\begin{prop}\cite{KPTT11}
Let $\cl{A}$ and $\cl{B}$ be unital $C^*$-algebras, then $\cl{A}\otimes_{\min} \cl B\subseteq_{\operatorname{c.o.i}}\cl{A}\otimes_{C^*-\min}\cl{B}$ and
$\cl{A}\otimes_{\max} \cl B \subseteq_{\operatorname{c.o.i}}\cl{A}\otimes_{C^*-\max}\cl{B}$, where the $\otimes_{C^*-\min}$, $\otimes_{C^*-\max}$ denote the
the tensor products in the $C^*$-algebra category. 
\end{prop}

\begin{prop}\cite{KPTT11}\label{ceuqalmax}
Let $\cl{A}$ be a unital $C^*$-algebra and $\cl{S}$ be an operator system, then
$\cl{S}\otimes_c\cl{A}=\cl{S}\otimes_{\max}\cl{A}$.
\end{prop}

\begin{prop}\cite{KPTT11}[Injectivity of the min tensor product]
The min tensor product is injective in the sense that for every choices of four operator systems $\cl{S}$ and $\cl{T}$, $\cl{S}_1$, $\cl{T}_1$ with inclusions $\cl{S}\subseteq_{\operatorname{c.o.i}}\cl{S}_1$ and 
$\cl{T}\subseteq_{\operatorname{c.o.i}}\cl{T}_1$, we have that 
\[ \cl{S}\otimes_{\min}\cl{T}\subseteq_{\operatorname{c.o.i}}
\cl{S}_1\otimes_{\min}\cl{T}_1  .\]
\end{prop}

\begin{prop}\cite{FP12}[Projectivity of the max tensor product]\label{completequotientmapmaxtensor}
 The max tensor product is projective in the following sense: Let $\cl{S}$, $\cl{T}$, $\cl{R}$ be operator systems and suppose $\psi:\cl{S}\to \cl{R}$ is a complete quotient map, then  the map
 $\psi\otimes  \operatorname{id}_{\cl{T}}:\cl{S}\otimes_{\max} \cl{T}\to \cl{R}\otimes_{\max}\cl{T}$ is also a complete quotient map.
\end{prop}

Henceforth, unless specified, we always assume $2\leq n<\infty$.

%{\bf Add discussion of operator system tensor products.}
%%%%%%%%%%%%%%%%%%%%%%%%%%%%%%%%%%%%%%%%%%%%%%%%%%%%%%%%%%%%%%%%%%%%%%%%%%%%%%%%
\section{Tensor Products and $C^*$-nuclearity of $\cl{S}_n$}
We begin this section with the following proposition which shows that $\cl{S}_n$ is not (min, max)-nuclear.
\begin{prop}
$\cl{S}_n$ is not (min, max)-nuclear, for $2\leq n\leq \infty$.
\end{prop}
\begin{proof}
It is known in that the operator system $\mathfrak{S}_1$ generated by a universal unitary is not (min, max)-nuclear because of the following \cite[Theorem 3.7]{FKPT14}:
\[ \mathfrak{S}_1\otimes_{\min}\mathfrak{S}_1\neq \mathfrak{S}_1\otimes_{\max}\mathfrak{S}_1. \]
On the other hand, we have that $\cl{S}_1=\mathfrak{S}_1$ \cite[Corollary 3.3]{ZHE14}, so we know that
\[ \cl{S}_1\otimes_{\min}\cl{S}_1\neq \cl{S}_1\otimes_{\max}\cl{S}_1.\]
Now, for $n\geq 2$, if $\cl{S}_n\otimes_{\min}\cl{S}_n=\cl{S}_n\otimes_{\max}\cl{S}_n$,
then \cite[Corollary 3.4, 3.5]{ZHE14} together imply that
\[ \cl{S}_1\otimes_{\min}\cl{S}_1=\cl{S}_1\otimes_{\max}\cl{S}_1, \]
which is a contradiction. Thus, $\cl{S}_n$ is not (min, max)-nuclear.
\end{proof}

%However, we can prove that $\cl{S}_n$ is $C^*$-nuclear for $2\leq n\leq \infty$ by using the nuclearity of $\cl{O}_n$. We first deal with the case when $2\leq n<\infty$.

\begin{lemma}\label{operatorsystemimpliesnuclearity}
For $2\leq n\leq \infty$, we assume that $\hat{\cl{S}}\subseteq \cl{O}_n$ is an operator system containing $\cl{S}_n$ and $\cl{A}$ be a $C^*$-algebra. If we have that
\[  \hat{\cl{S}}\otimes_{\min} \cl{A}=\hat{\cl{S}}\otimes_{\max} \cl{A},\]
then
\[ \cl{O}_n\otimes_{\min} \cl{A}=\cl{O}_n\otimes_{\max} \cl{A}.\]
\end{lemma}
\begin{proof}
We first represent $\cl{O}_n\otimes_{\max} \cl{A}$ on some Hilbert space $\cl{H}$. By the definition of the max tensor product of operator systems, the canonical embedding map from $\hat{\cl{S}}\otimes_{\max} \cl{A}$ into $\cl{O}_n\otimes_{\max} \cl{A}$ is completely positive. Thus we have a completely positive map $\rho:\hat{\cl{S}}\otimes_{\min} \cl{A}\to B(\cl{H})$, such that $\rho(a\otimes b)=a\otimes b$ for each $a\in \hat{\cl{S}}$ and $b\in \cl{A}$.

The injectivity of the min tensor product of operator systems implies that $\hat{\cl{S}}\otimes_{\min} \cl{A}\subseteq_{\operatorname{c.o.i}} \cl{O}_n\otimes_{\min} \cl{A}$, and we can extend $\rho$ to a completely positive map $\tilde{\rho}:\cl{O}_n\otimes_{\min} \cl{A}\to B(\cl{H})$ by  the Arveson's extension theorem.

Next, we use the Stinespring's dilation theorem and obtain a unital $*$-homomorphism $\gamma:\cl{O}_n\otimes_{\min} \cl{A}\to B(\cl{K})$ and $V:\cl{H}\to \cl{K}$ for some Hilbert space $\cl{K}$ such that
\[ \tilde{\rho}(a)=V^*\gamma(a)V, \text{ for each }a \in \cl{O}_n\otimes_{\min} \cl{A} . \]
The map $\rho$ being unital implies that $\tilde{\rho}$ is unital and hence $V^*V=I_{\cl{H}}$, i.e. $V$ is an isometry. By identifying $\cl{H}$ with $V\cl{H}$, we can assume that $\cl{H}\subseteq \cl{K}$.

Now, if we decompose $\cl{K}=\cl{H}+\cl{H}^{\perp}$, then $\tilde{\rho}$ is the $1,1$ corner of $\gamma$. Further, we have that
\[ \gamma(S_i\otimes 1_{\cl{A}})=\begin{pmatrix}
\tilde{\rho}(S_i\otimes 1_{\cl{A}})&C_i \\
B_i&D_i
\end{pmatrix}, \quad \text{for every } i\in \{1,\dots,n\}\]
Here, $B_i\in B(\cl{H},\cl{H}^{\perp})$, $C_i\in B(\cl{H}^{\perp},\cl{H})$, $D_i\in B(\cl{H}^{\perp},\cl{H}^\perp)$. Since $S_i\otimes 1_{\cl{A}}$ is
an isometry, it follows that $\gamma(S_i\otimes 1_{\cl{A}})$ and $\tilde{\rho}(S_i\otimes 1_{\cl{A}})$ are isometries, and we immediately have that $B_i=0$.

Moreover, the condition that $\sum_{i=1}^nS_iS_i^*=I_{\cl{H}}$ ($n$ finite) or $\sum_{i=1}^kS_iS_i^*\leq I_{\cl{H}}$ for every $1\leq k<\infty$ ($n$ infinite) implies that
\[ \sum_{i=1}^n\gamma(S_i\otimes 1_{\cl{A}})\gamma(S_i\otimes 1_{\cl{A}})^*=\gamma(\sum_{i=1}^nS_iS_i^*\otimes 1_{\cl{A}})=1_{\cl{K}}  ,\]
or 
\[ \sum_{i=1}^k\gamma(S_i\otimes 1_{\cl{A}})\gamma(S_i\otimes 1_{\cl{A}})^*=\gamma(\sum_{i=1}^kS_iS_i^*\otimes 1_{\cl{A}})\leq1_{\cl{K}}, \quad \text{for every } 1\leq k<\infty,\]
which means that
\[ \begin{pmatrix}
\sum_{i=1}^n\tilde{\cl{\rho}}(S_i\otimes 1_{\cl{A}})\tilde{\cl{\rho}}(S_i\otimes 1_{\cl{A}})^*+C_iC_i^*&\sum_{i=1}^nC_iD_i^* \\
\sum_{i=1}^nC_iC_i^* & \sum_{i=1}^nD_iD_i^*
\end{pmatrix}=1_{\cl{K}} ,\]
or
\[ \begin{pmatrix}
\sum_{i=1}^k\tilde{\cl{\rho}}(S_i\otimes 1_{\cl{A}})\tilde{\cl{\rho}}(S_i\otimes 1_{\cl{A}})^*+C_iC_i^*&\sum_{i=1}^kC_iD_i^* \\
\sum_{i=1}^kC_iC_i^* & \sum_{i=1}^kD_iD_i^*
\end{pmatrix}\leq1_{\cl{K}}, \quad \text{for every } 1\leq k<\infty .\]
Thus, we have that $C_i=0$ for every $i\in \{1,\dots,n\}$ and hence, 
\[ \gamma(S_i\otimes 1_{\cl{A}})=\begin{pmatrix}
\tilde{\rho}(S_i\otimes 1_{\cl{A}})&0 \\
0&D_i
\end{pmatrix} .\]
On the other hand, for each unitary $u\in \cl{A}$, similarly, we have that
\[ \gamma(I_{\cl{H}}\otimes u)=\begin{pmatrix}
\tilde{\rho}(I_{\cl{H}}\otimes u)&0 \\
0&v
\end{pmatrix}, \]
where $v$ is a unitary in $B(\cl{H}^{\perp})$.

Because $\cl{A}$ is spanned by its unitaries, and every $X\otimes z\in \cl{O}_n\otimes_{\min} \cl{A}$ can be written as $X\otimes z=
(X \otimes 1_{\cl{A}})(1_{\cl{H}}\otimes z)$, we see that $\gamma$ is diagonal on all elementary tensors. Then by continuity of $\gamma$,
we know it is diagonal on $\cl{O}_n\otimes_{\min} \cl{A}$.

We now have that the compression of $\gamma$ onto the $1,1$ corner is a $*$-homomorphism from $\cl{O}_n\otimes_{\min} \cl{A}$ to $B(\cl{H})$,
and this compression is exactly $\tilde{\rho}$. Moreover, $\tilde{\rho}(\hat{\cl{S}}\otimes_{\min} \cl{A})\subseteq \cl{O}_n\otimes_{\max} \cl{A}$.
Then, $\tilde{\rho}$ being a $*$-homomorphism implies that $\tilde{\rho}(\cl{O}_n\odot_{\min} \cl{A})\subseteq \cl{O}_n\otimes_{\max} \cl{A}$, where
$\odot_{\min}$ denotes the algebraic tensor product of $\cl{O}_n$ with $\cl{A}$ endowed with the minimal tensor norm. The continuity of $\tilde{\rho}$ implies
further that $\tilde{\rho}(\cl{O}_n\otimes_{\min} \cl{A})\subseteq \cl{O}_n\otimes_{\max} \cl{A}$. Form this, we can conclude that
$\tilde{\rho}(\cl{O}_n\otimes_{\min} \cl{A})=\cl{O}_n\otimes_{\max} \cl{A}$, because by the way $\tilde{\rho}$ is defined, $\Ran\tilde{\rho}$ is dense
in $\cl{O}_n\otimes_{\max} \cl{A}$.

Finally, $\tilde{\rho}(X\otimes z)=X\otimes z$ for every $X\otimes z\in \cl{O}_n\otimes_{\min} \cl{A}$ forces that the identity map from
$\cl{O}_n\odot_{\min} \cl{A}$ to $\cl{O}_n\otimes_{\max} \cl{A}$ extends to a $*$-homomorphism from $\cl{O}_n\otimes_{\min} \cl{A}$
onto $\cl{O}_n\otimes_{\max} \cl{A}$. Thus, $\cl{O}_n\otimes_{\min} \cl{A}=\cl{O}_n\otimes_{\max} \cl{A}$.
\end{proof}

Let $T_1,\dots,T_n$ be the generators of the Toeplitz-Cuntz algebra $\cl{TO}_n$ and $\cl{T}_n$ be the operator system generated by
$T_i$'s. By corollary 3.3 in \cite{ZHE14}, we know that $\cl{T}_n=\cl{S}_n$ via the natural isomorphism.
\begin{thm}\label{nuclearityofCuntz}
 Let $\cl{A}$ be a unital $C^*$-algebra, then we have that $\cl{S}_n\otimes_{\max}\cl{A}\subseteq_{\operatorname{c.o.i}}\cl{TO}_n\otimes_{\max}\cl{A}$, where $\cl{TO}_n$ is the Toeplitz-Cuntz algebra.
\end{thm}

Before proving this theorem, we need the following refined version of Bunce's result \cite[Proposition 1.]{BUN84}.
\begin{lemma}\label{Buncerefine}
 Let $(A_1,\dots,A_n)\in B(\cl{H})$ be a row contraction, then there exist isometric dilations $W_1,\dots,W_n\in B(\cl{K})$ of $A_1,\dots,A_n$ such that
 $W_i^*W_j=0$ for $i\neq j$, where $\cl{K}=\cl{H}\oplus(\oplus_{k=1}^{\infty}\cl{H}^{(n)})$. Moreover, $T_i$ can be chosen as the following form,
 \[ W_i=\begin{pmatrix}
         A_i&0\\
         X_i&YZ_i
        \end{pmatrix}
,\]
where the entries of $X_i$, $Y$ and $Z_i$ are all from $C^*(I, A_1,\dots,A_n)$.
 \end{lemma}
\begin{proof}
 The fact that the entries of $X_i$ and $Y$ are from $C^*(I, A_1,\dots,A_n)$ is directly from Bunce's construction. On the other hand, by his
 construction, $Z_i$'s can be any set of Cuntz isometries
 on $\oplus_{k=1}^{\infty}\cl{H}^{(n)}$ so we can choose a particular one as:
 \[ Z_k=(Z_{ij})=\begin{cases}
         I^{(n)}& \mbox{ if } i=(j-1)n+k\\
         0& \mbox{ else},
        \end{cases}
\]
where $I^{(n)}$ denotes the identity operator on $\oplus_{k=1}^{\infty}\cl{H}^{(n)}$.
\end{proof}
\begin{proof}[Proof of Theorem \ref{nuclearityofCuntz}]
By Proposition \ref{ceuqalmax}, we can show instead that $\cl{T}_n$ satisfies that
\[ \cl{T}_n\otimes_c\cl{A}\subseteq_{\operatorname{c.o.i}}\cl{TO}_n\otimes_{c}\cl{A}.\]
To this end, it is enough to show that for any pair of unital completely positive maps $\varphi: \cl{T}_n\to B(\cl{H})$ and $\psi:\cl{A}\to B(\cl{H})$ with commuting ranges, there always exists an extension $\tilde{\varphi}:\cl{TO}_n$ of $\varphi$ such
that the range of $\tilde{\varphi}$ and $\psi$ commute.

Since $\varphi$ is unitally completely positive, $(\varphi(T_1),\dots,\varphi(T_n))$ is a row contraction and hence can be dilated to isometries $W_1,\dots,W_n$ with
orthogonal ranges, by Lemma \ref{Buncerefine}. Then, there is a $*$-homomorphism $\pi:\cl{TO}_n\to B(\cl{K})$ such that $\pi(T_i)=W_i$. Meanwhile, we set
$\tilde{\psi}: \cl{R}\to B(\cl{K})$ as $\tilde{\psi}=\psi\oplus(\oplus_{k=1}^\infty\psi^{(n)})$, where $\psi^{(n)}$ denotes the direct sum of $n$ copies of $\psi$.

It is easy to see that $\tilde{\psi}$ and $\pi$ have commuting ranges. Clearly, $\psi= P_{\cl{H}}\tilde{\psi}|_{\cl{H}}$. Now, let $\tilde{\varphi}=P_{\cl{H}}\pi|_{\cl{H}}$, then it follows that $\tilde{\varphi}$ is a unital completely positive extension of $\varphi$ and $\psi$ and $\tilde{\varphi}$ has commuting ranges.
Thus, we have shown that $\cl{T}_n\otimes_c\cl{\cl{A}}\subseteq_{c.o.i}\cl{TO}_n\otimes_{c}\cl{A}$.
\end{proof}

%For the next three corollaries, we assume $2\leq n <\infty$. 
\begin{cor} Let $\cl A$ be a unital C*-algebra. If $\cl{TO}_n \otimes_{\min}
  \cl A = \cl{TO}_n \otimes_{\max} \cl A$, then $\cl S_n \otimes_{\min}
  \cl A = \cl S_n \otimes_{\max} \cl A.$
\end{cor}
\begin{proof}
By the nuclearity of $\cl{TO}_n$ and the injectivity of the min-tensor product, we have the following relations:
 \begin{align*}
  \cl{T}_n\otimes_{\min} \cl{A}\subseteq_{\operatorname{c.o.i}}\cl{TO}_n&\otimes_{\min}\cl{A}\\
  &\veq \\
    \cl{T}_n\otimes_{c=\max} \cl{A}\subseteq_{\operatorname{c.o.i}}\cl{TO}_n&\otimes_{c=\max}\cl{A}.
 \end{align*}
This implies that $\cl{T}_n\otimes_{\min} \cl{A}=  \cl{T}_n\otimes_{\max} \cl{A}$, so, equivalently, we know that $\cl{S}_n\otimes_{\min} \cl{A}=  \cl{S}_n\otimes_{\max} \cl{A}$

\end{proof}
\begin{cor}\label{OnimpliesSn}
Let $\cl A$ be a unital C*-algebra. If $\cl O_n \otimes_{\min}
  \cl A = \cl O_n \otimes_{\max} \cl A$, then $\cl S_n \otimes_{\min}
  \cl A = \cl S_n \otimes_{\max} \cl A.$
\end{cor}
\begin{proof} 
Since $\cl{O}_n=\cl{TO}_n/\bb{K}$ \cite[Proposition 3.1]{CUN77} ($\bb{K}$ denotes the algebra of compact operators), we have the following commuting diagram:
\[ \xymatrix{\bb{K}\otimes_{\max}\cl{A}\ar[d]\ar[r]&\cl{TO}_n\otimes_{\max}\cl{A}\ar[d]\ar[r]&\cl{O}_n\otimes_{\max}\cl{A}\ar[d]\\
\bb{K}\otimes_{\min}\cl{A}\ar[r]&\cl{TO}_n\otimes_{\min}\cl{A}\ar[r]&\cl{O}_n\otimes_{\min}\cl{A}.}
  \]
By assumption, we have that $\cl O_n \otimes_{\min}
  \cl A = \cl O_n \otimes_{\max} \cl A$. Also, we know that $\bb{K}$ is nuclear, so 
  $\bb{K} \otimes_{\min}
  \cl A = \bb{K} \otimes_{\max} \cl A$. This implies that the first and third vertical map in the above diagram are indeed isomorphisms. Hence, the second vertical map is also an
  isomorphism, that is, $\cl{TO}_n \otimes_{\min}
  \cl A = \cl{TO}_n \otimes_{\max} \cl A$. Now the conclusion follows from the above corollary.

\iffalse
Finally, \cite[Corollary 3.3]{zheng2014operator} tells us that $\cl{T}_n$ is unitally completely order isomorphic to $\cl{S}_n$, and therefore
the theorem is proved.
\fi
\end{proof}

Combining Corollary \ref{OnimpliesSn} with Lemma \ref{operatorsystemimpliesnuclearity}, we have

\begin{thm} Let $\cl A$ be a unital C*-algebra. Then $\cl O_n
  \otimes_{\min} \cl A= \cl O_n \otimes_{\max} \cl A$ if and only if
  $\cl S_n \otimes_{\min} \cl A = \cl S_n \otimes_{\max} \cl A.$
\end{thm}
Since $\cl O_n$ is nuclear, we immediately have the following corollary.
\begin{cor}
We have that the $\cl{S}_n$ is $C^*$-nuclear.
\end{cor}

\section{Equivalent Conditions of the $C^*$-Nuclearity of $\cl{S}_n$ and the Dual Row Contraction Version of Ando's Theorem}

%Finally, combining this result with Lemma ?? we have:

In this section, we prove some necessary and sufficient conditions for $\cl{O}_n\otimes_{\min}\cl{A}=\cl{O}_n\otimes_{\max}\cl{A}$. Recall that in \cite{ZHE14}, we denote $\cl{E}_n=\Span\{E_{00},\sum_{i=1}^nE_{ii},E_{i0},E_{0i}:1\leq i\leq n\}$, where $E_{ij}$'s are the matrix units in $M_{n+1}$, 
and we proved that 
\begin{thm}\cite{ZHE14}
The map $\phi:\cl{E}_n\to \cl{S}_n$ defined by the following:
\[ \phi(E_{i0})=\frac{1}{2}S_i,\ \phi(E_{0i})=\frac{1}{2}S_i^*,\ \phi(E_{00})=\frac{1}{2}I,\ \phi(\sum_{i=1}^nE_{ii})=\frac{1}{2}I,\quad 1\leq i\leq n \]
is a complete quotient map, that is,  $\cl{E}_n/J\cong \cl{S}_n$ completely order isomorphically, where $J:=\Ker \phi=\Span\{E_{00}-\sum_{i=1}^nE_{ii}\}$.
\end{thm}
%Moreover, $J$ is a completely order proximinal
%kernel \cite[Lemma 4.3]{ZHE14} in the sense of \cite{KPTT13}, which means that the positive %cone in $M_p(\cl{E}_n/J)$ is of the following form:
%\[ M_p^+(\cl{E}_n/J)=\{(x_{ij}+J):(x_{ij})\in M_p(\cl{E}_n)^+\}  .\]
The next proposition shows that 
the operators system $\cl{E}_n$ is $C^*$-nuclear. Before proving it,
let's recall a useful result \cite[Lemma1.7]{FKPT14}.
\begin{lemma}\label{PtensorQinmax}
 Let $\cl{S}$ and $\cl{T}$ be operator systems, then $u\in
  (\cl{S}\otimes_{\max}\cl{T})^+ $ if and only if for each $\epsilon>0$, there exist $(P^{\epsilon}_{ij})\in M_{k_{\epsilon}}(\cl{S})^+$, $(Q^{\epsilon}_{ij})\in M_{k_{\epsilon}}(\cl{T})^+$, such that 
\[ \epsilon 1_{\cl{S}}\otimes 1_{\cl{T}}+u=\sum_{i,j=1}^k P^{\epsilon}_{ij}\otimes Q^{\epsilon}_{ij} .\]
\end{lemma}

\begin{prop}\label{Enuclear}
We have that $\cl{E}_n\otimes_{\min}\cl{A}=\cl{E}_n\otimes_{\max}\cl{A}$ for every unital $C^*$-algebra $\cl{A}$.
\end{prop}
\begin{proof}
What we need to show is that $M_p(\cl{E}_n\otimes_{\min}\cl{A})^{+}\subseteq M_p(\cl{E}_n\otimes_{\max}\cl{A})^{+}$, for each $p\in \bb{N}$.
By the symmetry and associativity of the min and max tensor products of operator systems, and the nuclearity of $M_p$,  we have that
 \begin{align*}
M_p(\cl{E}_n\otimes_{\min}\cl{A})&=\cl{E}_n\otimes_{\min}M_p(\cl{A}) \\
M_p(\cl{E}_n\otimes_{\max}\cl{A})&=\cl{E}_n\otimes_{\max}M_p(\cl{A}).
\end{align*}
Notice that $M_p(\cl{A})$ is also a $C^*$-algebra, so it suffices to show that for each $A\in (\cl{E}_n\otimes_{\min}\cl{A})^{+}$, we have that $A\in (\cl{E}_n\otimes_{\max}\cl{A})^+$.

Since the $\min$ tensor product is injective, we have that $\cl{E}_n\otimes_{\min}\cl{A}\subseteq M_{n+1}\otimes_{\min} \cl{A}=M_{n+1}(\cl{A})$.
So for $A\in (\cl{E}\otimes_{\min}\cl{A})^{+}$, we know that it has the form
\[ A=\begin{pmatrix}
a_0&a_1&\cdots&a_n \\
a^*_1&b&&\\
\vdots&&\ddots&\\
a^*_n&&&b \\
\end{pmatrix}. \]
Without loss of generality, by considering $\epsilon I_{n+1}\otimes1_{\cl{A}}+A$,  we can assume that $a_0$ and $b$ are invertible.  According to Cholesky's lemma,
that
\begin{equation*} \begin{pmatrix}
a_0&a_1&\cdots&a_n \\
a^*_1&b&&\\
\vdots&&\ddots&\\
a^*_n&&&b \\
\end{pmatrix}=\begin{pmatrix}
a_0-\sum_{i=1}^na_ib^{-1}a_i^*&0&\cdots&0 \\
0&0&&\\
\vdots&&\ddots&\\
0&&&0
\end{pmatrix}+\begin{pmatrix}
\sum_{i=1}^na_ib^{-1}a^*_i&a_1&\cdots&a_n \\
a^*_1&b&&\\
\vdots&&\ddots&\\
a^*_n&&&b \\
\end{pmatrix}
\end{equation*}
The first matrix on the right side is positive in $M_{n+1}(\cl{A})$ and is easily seen to be in $(\cl{E}_n\otimes_{\max}\cl{A})^+$. What we need is to show that
the second matrix also lies in $(\cl{E}_n\otimes_{\max}\cl{A})^+$.

To this end, we use Proposition \ref{PtensorQinmax} and construct the following two matrices,
\[P=(P_{ij})=\begin{pmatrix}\begin{bmatrix}
0&0&\cdots&0 \\
0&1&&\\
\vdots&&\ddots&\\
0&&&1 \\
\end{bmatrix}& \begin{bmatrix}
0&0&\cdots&0 \\
1&0&&\\
\vdots&&\ddots&\\
0&&&0\\
\end{bmatrix}&\cdots&\begin{bmatrix}
0&0&\cdots&0 \\
0&0&&\\
\vdots&&\ddots&\\
1&&&0\\
\end{bmatrix}\\\begin{bmatrix}
0&1&\cdots&0 \\
0&0&&\\
\vdots&&\ddots&\\
0&&&0 \\
\end{bmatrix}&\begin{bmatrix}
1&0&\cdots&0 \\
0&0&&\\
\vdots&&\ddots&\\
0&&&0\\
\end{bmatrix}&&\\
\vdots&&\ddots&\\
\begin{bmatrix}
0&0&\cdots&1 \\
0&0&&\\
\vdots&&\ddots&\\
0&&&0\\
\end{bmatrix}&&&\begin{bmatrix}
1&0&\cdots&0 \\
0&0&&\\
\vdots&&\ddots&\\
0&&&0\\
\end{bmatrix}
\end{pmatrix}, \]

\[ Q=(Q_{ij})=\begin{pmatrix}
b&a^*_1&\cdots&a^*_n \\
a_1& &&  \\
\vdots&&B&  \\
a_n&&&
\end{pmatrix},  \]
where $B=(B_{ij})=(a_ib^{-1}a_j^*)$.

Then it is not hard to check that $P\in M_{n+1}(\cl{E}_n)^+$ and $Q\in M_{n+1}(\cl{A})^+$. Also, we have that
\[ \begin{pmatrix}
\sum_{i=1}^na_ib^{-1}a^*_i&a_1&\cdots&a_n \\
a^*_1&b&&\\
\vdots&&\ddots&\\
a^*_n&&&b \\
\end{pmatrix}=\sum_{i,j=1}^{n+1}P_{ij}\otimes Q_{ij}.   \]
This shows that the matrix on the left side is in $(\cl{E}_n\otimes_{\max}\cl{A})^+$, and  the lemma is proved.
\end{proof}

\begin{lemma}\label{positivelift}
Let $\cl{A}$ be a unital $C^*$-algebra, then we have that
$\operatorname{id}:\cl{S}_n\otimes_{\min}
\cl{A}\to\cl{S}_n\otimes_{\max} \cl{A}$ is completely positive  if and only if
 $\phi\otimes \operatorname{id}_{\cl{A}}:\cl{E}_n\otimes_{\min}\cl{A}\to \cl{S}_n\otimes_{\min}\cl{A}$ is a complete quotient map,
\end{lemma}
\begin{proof}
 We have the following diagram:
\[ \xymatrix{\cl{E}_n\otimes_{\min} \cl{A} \ar[d]^{\phi\otimes \operatorname{id}_{\cl{A}}} \ar[r]^{\cong}&\cl{E}_n\otimes_{\max} \cl{A}\ar[d]^{\phi\otimes \operatorname{id}_{\cl{A}}} \\
 \cl{S}_n\otimes_{\min} \cl{A}\ar[r]^{\operatorname{id}}& \cl{S}_n\otimes_{\max} \cl{A}}.  \]
 By Proposition \ref{completequotientmapmaxtensor}, we have that
$\phi\otimes \operatorname{id}_{\cl{A}}:\cl{E}_n\otimes_{\max}\cl{A}\to \cl{S}_n\otimes_{\max}\cl{A}$ is a complete quotient map, and hence
if $\phi\otimes \operatorname{id}_{\cl{A}}:\cl{E}_n\otimes_{\min}\cl{A}\to \cl{S}_n\otimes_{\min}\cl{A}$ is a complete quotient map, then every positive element $A$ in
$M_p(\cl{S}_n\otimes_{\min} \cl{A})$ has a positive pre-image in $M_p(\cl{E}_n\otimes_{\min} \cl{A})$, and therefore $A$ is in $M_p(\cl{S}_n\otimes_{\max} \cl{A})$.
So $\operatorname{id}:\cl{S}_n\otimes_{\min} \cl{A}\to\cl{S}_n\otimes_{\max} \cl{A}$ is completely positive. Conversely, if
$\operatorname{id}:\cl{S}_n\otimes_{\min} \cl{A}\to\cl{S}_n\otimes_{\max} \cl{A}$ is completely positive, then $\cl{S}_n\otimes_{\min} \cl{A}=\cl{S}_n\otimes_{\max} \cl{A}$.
Also, we have $\cl{E}_n\otimes_{\min} \cl{A}=\cl{E}_n\otimes_{\max} \cl{A}$. Thus,
$\phi\otimes \operatorname{id}_{\cl{A}}:\cl{E}_n\otimes_{\max}\cl{A}\to \cl{S}_n\otimes_{\max}\cl{A}$ being a complete quotient map means that
$\phi\otimes \operatorname{id}_{\cl{A}}:\cl{E}_n\otimes_{\min}\cl{A}\to \cl{S}_n\otimes_{\min}\cl{A}$ is a complete quotient map.
\end{proof}
The next theorem is now immediate:
\begin{thm} Let $\cl A$ be a unital C*-algebra. Then $\cl O_n
  \otimes_{\min} \cl A= \cl O_n \otimes_{\max} \cl A$ if and only if
$\phi \otimes id: \cl E_n \otimes_{\min} \cl A \to \cl S_n
\otimes_{\min} \cl A$ is a complete quotient map.
\end{thm}

We now prove a concrete condition on a unital C*-algebra that is equivalent
to the above condition. Given an operator system $\cl S$ we will write
{\bf $p >> 0$} provided that there exists $\epsilon >0$ such that $p -
\epsilon 1 \in \cl S^+$, and we set $\cl S^+_{-1}= \{ p \in \cl S: p >>0 \}.$ 

The reason for this notation is that if $\cl A$ is a unital C*-algebra and $\psi: \cl S \to \cl A$ is a
unital completely positive map, then  $p \in \cl S^+_{-1}$ implies $\psi(p)$ is positive and invertible in $\cl A.$ Moreover, $\cl S^+_{-1}$ is exactly the set of elements of $\cl S^+$ for which this is true for every unital completely positive map into a C*-algebra. Moreover, if $\psi: \cl T \to \cl S$ is a quotient map, then $p \in \cl S^+{-1}$ if and only if it has a pre-image, i.e., $\psi(r) =p$ with $r \in \cl T^+_{-1}$ (see \cite[Proposition 3.2]{FKP11}). 

\begin{thm}Let $\cl A$ be a unital C*-algebra. Then $\cl O_n
  \otimes_{\min} \cl A = \cl O_n \otimes_{\max} \cl A$ if and only if
  for all $p \in \bb N,$
whenever $I \otimes 1 + \sum_{j=1}^n S_j \otimes a_j + \sum_{j=1}^n
S_j^* \otimes a_j^* >>0$ in $\cl O_n \otimes_{\min} M_p(\cl A)$ then
there exists $a,b \in M_p(\cl A)^+_{-1}$ with $a+b = 1,$ such that
\[ \begin{bmatrix} a & a_1^* & \cdots &  a_n^*\\ a_1 & b & &
  \\ 
  \vdots &  & \ddots & \\ 
  a_n & &  & b   \end{bmatrix} \]
is in $M_{n+1}(M_p(\cl A))^+_{-1}.$
\end{thm}
\begin{proof} If $\cl O_n \otimes_{\min} \cl A = \cl O_n \otimes_{\max}
  \cl A$ then $q \otimes id: \cl E_n \otimes_{\min} M_p(\cl A) \to
  \cl S_n \otimes_{\min} M_p(\cl A)$ is a quotient map. Hence,
 $I \otimes 1 + \sum_{j=1}^n S_j \otimes a_j + \sum_{j=1}^n
S_j^* \otimes a_j^* \in q \otimes id \big( ( \cl E_n \otimes_{\min} M_p(\cl
A))^+ \big).$  Choosing any strictly positive element in the pre-image yields
the conclusion.

Conversely, the lifting formula shows that every element of the form
 $I \otimes 1 + \sum_{j=1}^n S_j \otimes a_j + \sum_{j=1}^n
S_j^* \otimes a_j^* >>0$ has a positive pre-image in $\cl E_n
\otimes_{\min} M_p(\cl A).$

Let $R=I \otimes r + \sum_{j=1}^n S_j
\otimes a_j + \sum_{j=1}^n S_j^* \otimes a_j^*$ be an arbitrary
element in $\cl S_n \otimes_{\min} M_p(\cl A)$ and let $\epsilon>0.$

Then
$T= I \otimes 1 + \sum_{j=1}^n S_j \otimes (r+\epsilon
1)^{-1/2}a_j(r+\epsilon)^{-1/2} + \sum_{j=1}^n S_j^* \otimes (r+
\epsilon 1)^{-1/2} a_j^*(r+ \epsilon 1)^{-1/2} >>0,$ and so by the
hypothesis has a lifting. Pre- and post-multiplying the entries of that lifting by
$(r+ \epsilon)^{1/2}$ gives a lifting of $R + \epsilon (I \otimes 1).$
This proves that the mapping $q \otimes id: \cl E_n \otimes_{\min}
M_p(\cl A) \to \cl S_n \otimes_{\min} M_p(\cl A)$ is a quotient map,
and since $p$ was arbitrary, this map is a complete quotient map.
\end{proof}

\begin{cor}\label{liftingcor} The C*-algebra $\cl O_n$ is nuclear if and only if
  whenever $\cl A$ is a unital C*-algebra and
$I \otimes 1 + \sum_{j=1}^n S_j \otimes a_j + \sum_{j=1}^n S_j^*
\otimes a_j^* >>0$ in $\cl O_n \otimes_{\min} \cl A$ there exists $a,b
\in \cl A^+_{-1}$ with $a+b=1$ such that
\begin{equation}\label{LP}  \begin{bmatrix} a & a_1^* & \cdots &  a_n^*\\ a_1 & b & &
  \\ 
  \vdots &  & \ddots & \\ 
  a_n & &  & b   \end{bmatrix} \end{equation}
is in $M_{n+1}(\cl A)^+_{-1}.$
\end{cor}
\begin{defn}
Let $\cl{A}$ be a unital $C^*$-algebra, then an $n$-tuple $(a_1,\dots,a_n)$ in $\cl{A}$ is called a \emph{dual row contraction} if 
\[ I \otimes 1 + \sum_{j=1}^n S_j \otimes a^*_j + \sum_{j=1}^n S_j^*
\otimes a_j \geq 0, \]
where the $S_i$'s are Cuntz isometries. Moreover, it is called a \emph{strict dual row contraction} if 
\[ I \otimes 1 + \sum_{j=1}^n S_j \otimes a^*_j + \sum_{j=1}^n S_j^*
\otimes a_j >>0. \]
\end{defn}
\begin{remark}
Note that a dual row contraction is a row contraction, since
\[ I \otimes 1 + \sum_{j=1}^n S_j \otimes a^*_j + \sum_{j=1}^n S_j^*
\otimes a_j \geq 0 \]
implies that 
\[ I \otimes 1 + z\sum_{j=1}^n S_j \otimes a^*_j + \bar{z}\sum_{j=1}^n S_j^*
\otimes a_j \geq 0, \quad \text{for all } z\in \bb{T}  ,\]
which is equivalent to 
\[ w(\sum_{j=1}^n S_j \otimes a^*_j)\leq \frac{1}{2} ,\]
where $w$ means the numerical radius. 
So, we have that
\[ \biggl\|\sum_{i=1}^na_ia_i^*\biggr\|=\biggl\|\biggl(\sum_{i=1}^nS_i\otimes_{\min}a_i^*\biggr)^*
\biggl(\sum_{i=1}^nS_i\otimes_{\min}a_i^*\biggr)\biggr\|\leq
(2 w(\sum_{j=1}^n S_j \otimes a^*_j))^2\leq 1.\]
But not every row contraction is a dual row contraction. A counter-example can be easily constructed. In particular, $n$ ($2\leq n<\infty$) Cuntz isometries form a row contraction but not dual row contraction, since $\sum_{i=1}^nS_i\otimes S_i^*$ is a unitary whose spectrum is the whole unit circle. 
\end{remark}

Again, since $\cl{O}_n$ is nuclear, Corollary \ref{liftingcor} is indeed a (strict) dual row contraction version of Ando's theorem (See \cite{AND73} for the original version). Moreover, when $\cl{M}$ is a von Neumann algebra, we can replace "strict dual row contraction" by "dual row contraction" and "strictly positive" by "positive" by taking weak*-limits. We summarize these statements below. This result is a dual row contraction version of Ando's theorem on numerical radius.
 
\begin{thm}\label{strongrowcontractionando}
Let  $\cl A$ be a unital C*-algebra and
$(a_1,\dots,a_n)\in \cl{A}$ be a strict dual row contraction, then there exists $a,b
\in \cl A^+_{-1}$ with $a+b=1$ such that
\begin{equation*}  \begin{bmatrix} a & a_1 & \cdots &  a_n\\ a_1^* & b & &
  \\ 
  \vdots &  & \ddots & \\ 
  a_n^* & &  & b   \end{bmatrix} \end{equation*}
is in $M_{n+1}(\cl A)^+_{-1}.$
Moreover, if $\cl{M}$ is a von Neumann algebra and $(a_1,\dots,a_n)\in \cl M$ is a dual row contraction, then there exists $a,b
\in \cl M^+$ with $a+b=1$ such that
\begin{equation*} \begin{bmatrix} a & a_1& \cdots &  a_n\\ a_1^* & b & &
  \\ 
  \vdots &  & \ddots & \\ 
  a_n^* & &  & b   \end{bmatrix} \end{equation*}
is in $M_{n+1}(\cl M)^+.$
\end{thm}

\section{An Alternative Proof of the Nuclearity of $\cl{O}_n$}
We now  give anew proof of the nuclearity of $\cl O_n$,
by showing directly the existence of operators $a,b$ mentioned in Corollary \ref{liftingcor}, which will prove that $\cl O_n$ is nuclear.

To this end, we shall need the notion of "shorted operators", which was introduced in \cite{AT75}. Here, we briefly quote some results we will need in our proof.
\begin{defn}\cite{AT75}
Let $\cl{H}$ be a Hilbert space and $A\in B(\cl{H})$. Assume $S\subseteq \cl{H}$ is a closed subspace, then the \emph{shorted operator of} $A$ with respect to $S$, denoted as $S(A)$ is defined as the maximum of the following set:
\[ \{T\in B(\cl{H}):T\leq A, \Ran T\subseteq S \}  .\]
Also, we denote $S_0(A)=S(A)|_S$. 
\end{defn}

 The shorted operator always exists \cite[Theorem 1]{AT75}. Moreover, we have that
 \begin{prop}\cite{AT75}\label{shortedinner}
 For each $x\in S$, we have that 
\[ \inner{S_0(A)x}{x}=\inf\biggl\{\Inner{A\begin{pmatrix}x\\y\end{pmatrix}}{\begin{pmatrix}x\\y\end{pmatrix}}:y\in S^{\perp} \biggr\}  .\]
\end{prop}
We now prove  that the condition of Corollary \ref{liftingcor} is met for $n=2$ without using the nuclearity of $\cl{O}_2$. 
\begin{proof}[A proof of the nuclearity of $\cl{O}_2$]
%To simplify our discussion, we will assume $n=2$, and other cases follow similarly. 

 Let $\cl{A}\subseteq B(\cl{H})$ be a unital $C^*$-algebra and let $(a_1,\dots,a_n)\in \cl{A}$ be a strict dual row contraction, that is, 
\[ A=I \otimes 1 + \sum_{j=1}^2 S_j \otimes a^*_j + \sum_{j=1}^2 S_j^*
\otimes a_j >>0,  \text{ in }\cl O_2 \otimes_{\min} \cl A .\] By Corollary 3.3 in \cite{ZHE14} the operator system spanned by the Toeplitz-Cuntz isometries is completely order isomorphic to the operator system spanned by the Cuntz isometries.  Thus, we can take the $S_i$'s to be Toeplitz-Cuntz isometries. Moreover, it suffices to consider the following specific choice of Toeplitz-Cuntz isometries: 
\[ S_i\in B(l^2), \quad S_i(e_k)=e_{kn+i},\quad k=0,1,2,\dots, \quad i=1,2 ,\]
where $\{e_i:i=0,1,2,\dots\}$ is the orthonormal basis of $l^2$.

We write $l^2 \otimes \cl H = \oplus_{i=0}^{+\infty} \cl H_i,$ where $\cl H_i = \cl H$ for all $i$.
Thus, $A$ corresponds to the following operator in $B(l^2 \otimes \cl H)$, 
\[ A=\begin{pmatrix}
1&a_1&a_2&0&0&0&0&0&0&\cdots \\
a^*_1&1&0&a_1&a_2&0&0&0&0&\cdots\\
a^*_2&0&1&0&0&a_1&a_2&0&0&\cdots\\
0&a_1^*&0&1&0&0&0&a_1&a_2&\cdots \\
0&a_2^*&0&0&1&0&0&0&0&\cdots\\
0&0&a_1^*&0&0&1&0&0&0&\cdots\\
0&0&a_2^*&0&0&0&1&0&0&\cdots\\
0&0&0&a_1^*&0&0&0&1&0&\cdots\\
0&0&0&a_2^*&0&0&0&0&1&\cdots\\
\vdots&\vdots&\vdots&\vdots&\vdots&\vdots&\vdots&\vdots&\vdots&\\
\end{pmatrix}.
\]
Set $\cl R_k = \oplus_{i=k}^{+\infty} \cl H_i.$
We then write $A$ as the following block form:
\[ A=\begin{pmatrix}
A_{11}&A_{12}\\
A_{21}&A_{22}
\end{pmatrix},\]
where $A_{11}\in B(\cl{H}_0)$, $A_{12}\in B(\cl R_1,\cl{H}_0)$, $A_{21}\in B(\cl{H}_0,\cl R_1)$, $A_{22}\in B(\cl R_1)$.\\
Now, let $B=\cl{H}_0(A)$, then by Proposition \ref{shortedinner}, we have that 
\begin{align*}
\inner{Bh_0}{h_0}&=\inf_{g\in \cl R_1}\biggl\{\Inner{A\begin{pmatrix}
h_0\\
g
\end{pmatrix}}{\begin{pmatrix}
h_0\\
g
\end{pmatrix}}\biggr\}\\
&=\inf_{h_1\in\cl{H}}\inf_{h_2\in\cl{H}}\inf_{z\in \cl R_3}\biggl\{\Inner{A\begin{pmatrix}
h_0\\
h_1\\
h_2\\
z
\end{pmatrix}}{\begin{pmatrix}
h_0\\
h_1\\
h_2\\
z
\end{pmatrix}}\biggr\} \\
&=\inf_{h_1\in\cl{H}}\inf_{h_2\in\cl{H}}\inf_{z\in \cl R_3}\biggl\{\inner{h_0+a_1h_1+a_2h_2}{h_0}+\inner{a_1^*h_0}{h_1}+\inner{a_2^*h_0}{h_2}+\Inner{A_{22}\begin{pmatrix}
h_1\\
h_2\\
z
\end{pmatrix}}{\begin{pmatrix}
h_1\\
h_2\\
z
\end{pmatrix}}\biggr\}
\end{align*}
We claim that
\[ \inf_{z \in \cl R_3} \biggl\{ \Inner{A_{22}\begin{pmatrix}
h_1\\
h_2\\
z
\end{pmatrix}}{\begin{pmatrix}
h_1\\
h_2\\
z
\end{pmatrix}}\biggr\} = \inner{Bh_1}{h_1} + \inner{Bh_2}{h_2} .\]

Assuming this claim for the moment, we have
\begin{align*} \inner{Bh_0}{h_0} =& 
\inf_{h_1\in\cl{H}}\inf_{h_2\in\cl{H}}\{\inner{h_0}{h_0}+\inner{a_1f}{h_0}+\inner{a_2h_2}{h_0}+\inner{a_1^*h_0}{h_1}+\inner{a_2^*h_0}{h_2}\\
&+\inner{Bh_1}{h_1}+\inner{Bh_2}{h_2} \} \\
 =&\inf_{h_1\in\cl{H}}\inf_{h_2\in\cl{H}}\biggl\{\Inner{\begin{pmatrix}
1&a_1&a_2\\
a_1^*&B&0 \\
a_2^*&0&B
\end{pmatrix}\begin{pmatrix}
h_0\\
h_1\\
h_2
\end{pmatrix} }{\begin{pmatrix}
h_0\\
h_1\\
h_2
\end{pmatrix}}\biggr\}.
\end{align*}
So we have that 
\[ \begin{pmatrix}
1-B&a_1&a_2\\
a_1^*&B&0 \\
a_2^*&0&B
\end{pmatrix}\geq 0.\]

To justify the claim, we write $\bb N$ as the disjoint union of $N_1 = \{1+2(2^k-1), 1+3(2^k-1) : k \ge 0\}$ and $N_2= \{ 2+3(2^k-1). 2+4(2^k-1): k \ge 0 \}$. Set  $\cl N_k = \oplus_{i \in N_k} \cl H_i$ so that $\cl R_1 = \cl N_1 \oplus \cl N_2$. Observe that both of these subspaces are reducing for $A_{22}$ and that with respect to the obvious identification of $\cl N_k \sim \oplus_{i=0}^{+\infty} \cl H_i$ we have that $A_{22} \sim A \oplus A$.

%the restriction   Relative to this decomposition, we have that $\cl R_1 = (\oplus_{i \in N_1} \cl H_i) %\oplus (\oplus_{i \in N_2} \cl H_i).$   we can write $A_{22}=A_{22}^1+A_{22}^2$, where 
%\[ (A_{22}^1)_{ij}=\begin{cases}
%(A_{22})_{ij},\quad i,j=1+2(2^k-1),\dots,1+3(2^k-1), k=0,1,2,\dots\\
%0, \quad \text{else}
%\end{cases}, \]
%and
%\[    (A_{22}^2)_{ij}=\begin{cases}
%(A_{22})_{ij},\quad i,j=2+3(2^k-1),\dots,2+4(2^k-1), k=0,1,2,\dots\\
%0, \quad \text{else}
%\end{cases}                  \]
 %Here $(A_{22})_{ij}$ denotes the $i,j$-th entry of $A_{22}$, same for $(A_{22}^1)_{ij}$  and $%(A_{22}^2)_{ij}$. 
 
 Hence, 
 \begin{align*}
\inf_{z\in \cl R_3}\Inner{A_{22}\begin{pmatrix}
h_1\\
h_2\\
z
\end{pmatrix}}{\begin{pmatrix}
h_1\\
h_2\\
z
\end{pmatrix}}&=\inf_{z_1\in \cl N_1 \ominus \cl H_1}\Inner{A_{22}\begin{pmatrix}
h_1\\
0\\
z_1
\end{pmatrix}}{\begin{pmatrix}
h_1\\
0\\
z_1
\end{pmatrix}}+\inf_{z_2\in \cl N_2 \ominus \cl H_2}\Inner{A_{22}\begin{pmatrix}
0\\
h_2\\
z_2
\end{pmatrix}}{\begin{pmatrix}
0\\
h_2\\
z_2
\end{pmatrix}}\\
&=\inf_{z\in \cl R_1}\Inner{A\begin{pmatrix}
h_1\\
z
\end{pmatrix}}{\begin{pmatrix}
h_1\\
z
\end{pmatrix}}+
\inf_{z\in \cl R_1}\Inner{A\begin{pmatrix}
h_2\\
z
\end{pmatrix}}{\begin{pmatrix}
h_2\\
z
\end{pmatrix}}\\
&=\inner{Bh_1}{h_1}+\inner{Bh_2}{h_2}.
 \end{align*}

It remains to show that $B\in \cl{A}$.  Since $-S_i$'s are also Cuntz isometries, we have that
\[I \otimes 1 - \sum_{j=1}^2 S_j \otimes a^*_j - \sum_{j=1}^2 S_j^*
\otimes a_j >>0 .\]
It follows that $\|\sum_{j=1}^2 S_j \otimes a^*_j - \sum_{j=1}^2 S_j^*\otimes a_j\|<1$, and therefore $\|1-A_{22}\|<1$. According to the proof of \cite[Theorem 1]{AT75}, the shorted operator $B$ has an explicit formula: $B=A_{11}-A_{12}A_{22}^{-1}A_{21}$. So what left for us to show is that all the entries of $A_{22}^{-1}$ are in $\cl{A}$. To see this, we first use the Neumann series to write $A_{22}^{-1}=\sum_{n=0}^\infty(1-A_{22})^n$. Since each row and column of $1-A_{22}$ only has finitely many nonzero entries, we must have that the entries of $(1-A_{22})^n$ are in $\cl{A}$ for each $n\in \bb{N}$. Since the Neumann series is norm convergent, we have that each entry of $A_{22}^{-1}$ is in $\cl A$ and  since $A_{12}$ and $A_{21}$ are only non-zero in finitely many entries, $B \in \cl{A}$.  

Finally, we can repeat the above process for $A-\epsilon 1\otimes 1>> 0$ and see that we can make both $B$  and $1-B$ strictly  positive with
\[ \begin{pmatrix}
1-B&a_1&a_2\\
a_1^*&B&0 \\
a_2^*&0&B
\end{pmatrix}>>0 .\]
\end{proof}

%Next, suppose $\cl M$ is a von Neumann algebra and 
%\[ A=1 \otimes 1 + \sum_{j=1}^2 S_j \otimes a^*_j + \sum_{j=1}^2 S_j^*
%\otimes a_j \geq 0,  \text{ in }\cl O_n \otimes_{\min}\cl M.\]
%Just as what we have done above, we write $A$ as an infinite matrix. Then, again by \cite[Theorem 1]{AT75}, the shorted operator $B=H_0(A)$ lies in $M$. Now, mimicking the method used above, we can see that $B$ does the job.  
%\end{proof}
%%%%%%%%%%%%%%%%%%%%%%%%%%%%%%%%%%%%%%%%%%%%%%%%%%%%%%%%%%%%%%%%%%%%%%%%%%%%%%%%%%%%%%%%
%\hrule
%\textbf{The General Case:}

The proof that $\cl O_n$ for $n \ge 3$ is nuclear can be done similarly and we only sketch the key points. Let $\cl{A}\subseteq B(\cl{H})$ be a unital $C^*$-algebra and $(a_1,\dots,a_n)\in \cl{A}$ be a strict dual row contraction, that is, 
\[ A=I\otimes 1 + \sum_{j=1}^n S_j \otimes a^*_j + \sum_{j=1}^n S_j^*
\otimes a_j >>0,  \text{ in }\cl O_n \otimes_{\min} \cl A .\] Then, by Corollary 3.3 in \cite{ZHE14}, we can take $S_i$'s as Toeplitz-Cuntz isometries. Moreover, it suffices to consider the following specific choice of Toeplitz-Cuntz isometries: 
\[ S_i\in B(l^2), \quad S_i(e_k)=e_{kn+i},\quad k=0,1,2,\dots, \quad i=1,2,\dots,n ,\]
where $\{e_i:i=0,1,2,\dots\}$ is an orthonormal basis of $l^2$.
Thus, $A$ corresponds to the following operator on $B(\cl{H}^{(\infty)})$, 
\[ A=\begin{pmatrix}
1&a_1&\cdots&a_n& &&\\
a^*_1&1&&&a_1&\cdots&a_n&\\
\vdots&&\ddots&&&&\\
a^*_n&&&1&&&&\\
&a^*_1&&&1&&&&\\
&\vdots&&&&\ddots&\\
&a^*_n&&&&&1&\\
&&&&&&&\ddots

\end{pmatrix}.
\]
Again we  write $A$ as the  block form:
\[ A=\begin{pmatrix}
A_{11}&A_{12}\\
A_{21}&A_{22}
\end{pmatrix},\]
where $A_{11}=1\in B(\cl{H})$, $A_{12}\in B(\cl{H}^{(\infty)}\ominus \cl{H},\cl{H})$, $A_{21}\in B(\cl{H},\cl{H}^{(\infty)}\ominus \cl{H})$, $A_{22}\in B(\cl{H}^{(\infty)})$.

Again we take $B= \cl H_0(A),$ the short of the operator $A$ to the 0-th subspace.
The calculation of $B$ proceeds as before, and in this case one shows that $\bb N$ decomposes into a disjoint union of $n$ subsets, $\bb N= N_1 \cup \cdots \cup N_n$ such that $\cl N_j= \oplus_{i \in N_j} \cl H_j$ is a reducing subspace for $A_{22}$ with the restriction to $\cl N_j$ of $A_{22}$ unitarily equivalent to $A$.

We leave the remaining details to the interested reader.

%We end this section with a little discussion of the nuclearity of $n=\infty$. 
%Notice that, with some tiny modifications, one can easily see that Lemma \ref{operatorsystemimpliesnuclearity} also works for $\cl{S}_\infty$. 

Having an alternative proof of the nuclearity of $\cl{O}_n$ ($2\leq n<\infty$), we can give an alternative proof of the nuclearity of $\cl{O}_\infty$ with just a little effort. Here is the proof. 
\begin{thm}
$\cl{O}_\infty$ is nuclear.
\end{thm}
\begin{proof}
Using Lemma \ref{operatorsystemimpliesnuclearity}, we just need to show that $\cl{S}_\infty$ is $C^*$-nuclear. 

It suffices to show that  $(\cl{S}_\infty\otimes_{\min}\cl{A})^+= (\cl{S}_\infty\otimes_{\max}\cl{A})^+$ for every unital $C^*$-algebra $\cl{A}$.

To see this, we choose $A\in (\cl{S}_\infty\otimes_{\min}\cl{A})^+$, then $A$ has the form,
\[ A= I\otimes X+\sum_{i\in F}S_i\otimes X_i+\sum_{i\in F}S^*_i\otimes X_i^*,\quad X_i\in \cl{A} , \]
where $F$ is a finite subset of $\bb{N}$. So there exists $N\in \bb{N}$, such that $F\subseteq \{1,\dots,N\}$. This means, by the injectivity the $\min$ tensor product, we
have that $A\in (\cl{S}_N\otimes_{\min}\cl{A})^+$.

But we have just shown that $\cl{S}_n$ is $C^*$-nuclear for $n$ finite, so we have that
\begin{align*}
 A\in (\cl{S}_N\otimes_{\min}\cl{A})^+
  =(\cl{S}_N\otimes_{\max}\cl{A})^+
  \subseteq(\cl{S}_\infty\otimes_{\max}\cl{A})^+.
\end{align*}
Thus, we know that $\cl{S}_\infty\otimes_{\min}\cl{A}=\cl{S}_\infty\otimes_{\max}\cl{A}$.
\end{proof}

\section{The dual operator system of $\cl{S}_n$}
In this section, we prove some properties of the dual operator system of $\cl{S}_n$, denoted by $\cl{S}_n^d$, which is the operator system consisting of all (bounded) linear
functionals on $\cl{S}_n$. \\
First, we choose a basis for $\cl{S}_n^d$ as the following,
\[\{ \delta_0, \delta_i,\delta_i^*: 1\leq i\leq n \}, \]
where
\begin{align*}
  \delta_0(I):=1, \delta_0(S_i):=\delta_0(S_i^*)=0, \quad \text{for all $i$} ; \\
  \delta_i(I)=0, \delta_i(S_j)=\delta_{ij}, \delta_i(S_k^*)=0 , \quad \text{for all $k$}; \\
  \delta^*_i(I)=0, \delta^*_i(S^*_j)=\delta_{ij}, \delta_i(S_k)=0 , \quad \text{for all $k$},
\end{align*}
where $\delta_{ij}$ is the Kronecker delta notation. So we have $\cl{S}_n^d=\Span\{ \delta_0, \delta_i,\delta_i^*: 1\leq i\leq n \}$.

Then, we define an order structure on $\cl{S}_n^d$ by
\[ (f_{ij})\in M_p(\cl{S}_n^d)^+ \Leftrightarrow (f_{ij}):\cl{S}_n\to M_p \text{ is completely positive }  . \]
It is a well-known result by Choi and Effros \cite[Theorem 4.4]{CE77} that with the order structure defined above,  the dual space of a finite dimensional operator system is again
an operator system with an Archimedean order unit, and indeed,  any strictly positive linear functional is an Archimedean order unit.

We claim that $\delta_0$ is strictly positive. To see this suppose that $p \in \cl S_n^+$ with $\delta_0(p) =0.$  Then $p = \sum_{i=1}^n a_i S_i + \sum_{i=1}^n \overline{a_i} S_i^*,$  Using the fact that, if $S_i$ are Cuntz isometries, then $-S_i$ are also Cuntz isometries, we see that $-p \in \cl S_n^+.$  Thus, $p=0.$

Hence, $\cl{S}_n^d$ is an operator system with Archimedean order unit $\delta_0$.

The following characterizes  positive elements in $M_p(\cl{S}_n^d)$  of the form  $I_p\otimes \delta_0+\sum_{i=1}^nA_i\otimes \delta_i+\sum_{i=1}^nA_i^*$.
\begin{prop}\label{positiveindual}
 An element $I_p\otimes \delta_0+\sum_{i=1}^nA_i\otimes \delta_i+\sum_{i=1}^nA_i^*\otimes \delta_i^*\in M_p(\cl{S}_n^d)$ is positive if and only if
 $(A_1,\dots,A_n)$ is a row contraction.
\end{prop}
\begin{proof}
 Let $M=I_p\otimes \delta_0+\sum_{i=1}^nA_i\otimes \delta_i+\sum_{i=1}^nA_i^*\otimes \delta_i^*$, and view $M$ as a completely positive map
 from $\cl{S}_n$ to $M_p$, it satisfies $M(I)=I_p$, $M(S_i)=A_i$. Thus, since $M$ us unitally completely positive, we have that  $(A_1,\dots,A_n)$ is a row contraction.\\
 Conversely, if $(A_1,\dots,A_n)$ is a row contraction, then there exists  a unital completely positive map which sends $S_i$ to $A_i$, $S_i^*$ to $A^*_i$, by Theorem.
 But this map is necessarily $M:\cl{S}_n \to M_p$, and this means $M\in M_p(\cl{S}_n^d)^+$.
 \end{proof}

 \begin{prop}
  Let $\cl{A}$ be a unital $C^*$-algebra and  $\phi:\cl{S}_n^d\to \cl{A}$ be a unital linear map. Then $\phi$is completely positive if and only if $\phi$ is self-adjoint
 and
  \[ w(A_1\otimes \phi(\delta_1)+\cdots+A_n\otimes \phi(\delta_n))\leq \frac{1}{2}, \]
  for each row contraction $(A_1,\dots,A_n)\in M_p$, each $p\in \bb{N}$, where $w$ denotes the numerical radius.
 \end{prop}
\begin{proof}
 Suppose $\phi$ is unitally completely positive, then for
 \[ M=I_p\otimes \delta_0+\sum_{i=1}^nA_i\otimes \delta_i+\sum_{i=1}^nA_i^*\otimes \delta_i^*\in M_p(\cl{S}_n^d)^+, \]
 we must have
 \begin{align*}
  I_p\otimes \phi(M)&=I_p\otimes I+\sum_{i=1}^nA_i\otimes \phi(\delta_i)+\sum_{i=1}^nA_i^*\otimes \phi(\delta_i^*) \\
  &=I_p\otimes I+\sum_{i=1}^nA_i\otimes \phi(\delta_i)+\sum_{i=1}^nA_i^*\otimes \phi(\delta_i)^* \geq 0.
 \end{align*}
By Proposition \ref{positiveindual}, $M$ is positive if and only if $(A_1,\dots,A_n)$ is a row contraction. Noting that $(zA_1,\dots,zA_n)$ is also a row contraction, we then have
that
\[I_p\otimes I+z\sum_{i=1}^nA_i\otimes \phi(\delta_i)+\bar{z}\sum_{i=1}^nA_i^*\otimes \phi(\delta_i)^*\geq 0, \]
which means that $w(\sum_{i=1}^nA_i\otimes \phi(\delta_i))\leq \frac{1}{2}$, for each row contraction $(A_1,\dots,A_n)\in M_p$ and each $p\in \bb{N}$.

Conversely, we suppose $w(\sum_{i=1}^nA_i\otimes \phi(\delta_i))\leq \frac{1}{2}$, for each row contraction $(A_1,\dots,A_n)\in M_p$ and each $p\in \bb{N}$,  and this implies that
\[ I_p\otimes I+\sum_{i=1}^nA_i\otimes \phi(\delta_i)+\sum_{i=1}^nA_i^*\otimes \phi(\delta_i)^*\geq 0, \]
for each row contraction $(A_1,\dots,A_n)\in M_p$ and each $p\in \bb{N}$.

Choose an arbitrary $N=B_0\otimes \delta_0+\sum_{i=1}^nB_i\otimes \delta_i+\sum_{i=1}^nB_i^*\otimes \delta_i^*\in M_p(\cl{S}_n^d)^+$, then for each $\epsilon>0$,
\begin{align*}
 \epsilon I_p\otimes \delta_0+N&=(\epsilon I_p+B_0)\otimes \delta_0+\sum_{i=1}^nB_i\otimes \delta_i+\sum_{i=1}^nB_i^*\otimes \delta_i^*\geq 0,
\end{align*}
which implies
\[I_p\otimes \delta_0+\sum_{i=1}^n (\epsilon I_p+B_0)^{-\frac{1}{2}}B_i(\epsilon I_p+B_0)^{-\frac{1}{2}}\otimes \delta_i+
\sum_{i=1}^n(\epsilon I_p+B_0)^{-\frac{1}{2}}B_i^*(\epsilon I_p+B_0)^{-\frac{1}{2}}\otimes \delta_i^*\geq 0.
\]
By Proposition \ref{positiveindual}, we have that $((\epsilon I_p+B_0)^{-\frac{1}{2}}B_1(\epsilon I_p+B_0)^{-\frac{1}{2}},\dots,
(\epsilon I_p+B_0)^{-\frac{1}{2}}B_n(\epsilon I_p+B_0)^{-\frac{1}{2}})$ is a row contraction, and therefore
\[  I_p\otimes I+\sum_{i=1}^n(\epsilon I_p+B_0)^{-\frac{1}{2}}B_i(\epsilon I_p+B_0)^{-\frac{1}{2}}\otimes \phi(\delta_i)+
\sum_{i=1}^n(\epsilon I_p+B_0)^{-\frac{1}{2}}B_i^*(\epsilon I_p+B_0)^{-\frac{1}{2}}\otimes \phi(\delta_i)^*\geq 0.   \]
Thus,
\[ \phi(\epsilon I_p\otimes \delta_0+N)\geq 0, \quad \text{ for each } \epsilon>0 .\]
So we have that $ \phi(N)\geq 0$, and this completes the proof.
\end{proof}

\begin{remark}
 Since compressions of row contractions are still row contractions, it follows that if
 \[w(A_1\otimes \phi(\delta_1)+\cdots+A_n\otimes \phi(\delta_n))\leq \frac{1}{2}, \]
  for each row contraction $(A_1,\dots,A_n)\in M_p$, each $p\in \bb{N}$, then for Cuntz isometries $S_1,\dots,S_n$,
  \[ w(S_1\otimes\phi(\delta_1)+\cdots+S_n\otimes \phi(\delta_n))\leq \frac{1}{2},\]
  where $S_i$'s are Cuntz isometries and the tensor product is the minimal one  so that
$S_1\otimes a_1^*+\cdots+S_n\otimes a_n^*\in \cl{O}_n\otimes_{\min}\cl{A}$.
  Conversely, using the universal property of $\cl{S}_n$, we have that  for each row contraction $(A_1,\dots,A_n)\in M_p$, the map sending $S_i$ to $A_i$, $S_i^*$ to $A_i^*$,
  $I$ to $I_p$ is completely positive and hence
    \[ w(S_1\otimes \phi(\delta_1)+\cdots+S_n\otimes\phi(\delta_n))\leq \frac{1}{2} \]
    implies that
    \[w(A_1\otimes \phi(\delta_1)+\cdots+A_n\otimes \phi(\delta_n))\leq \frac{1}{2}. \]
\end{remark}
On the other hand, we have that 
\[ w(S_1\otimes \phi(\delta_1)+\cdots+S_n\otimes \phi(\delta_n))\leq \frac{1}{2}\] if and only if 
\[ I \otimes 1 + \sum_{j=1}^n S_j \otimes a^*_j + \sum_{j=1}^n S_j^*
\otimes a_j \geq 0.\]
So we have proved the following corollary.
\begin{cor}\label{cpstrongrowcontraction}
A unital linear map $\phi:\cl{S}_n^d\to \cl{A} $ is completely positive if and only if $\phi$ is self-adjoint
 and
  \[ w(S_1\otimes \phi(\delta_1)+\cdots+S_n\otimes \phi(\delta_n))\leq \frac{1}{2}, \]
  where $S_1,\dots,S_n$ are Cuntz isometries if and only if 
  $(\phi(\delta_1)^*,\dots,\phi(\delta_n)^*)$ is a dual row contraction. 
\end{cor}
In \cite{POP09}, the joint numerical radius for $n-$tuple of operators $(T_1,\dots,T_n)\in B(\cl{H})$ is defined as:
\[ w(T_1,\dots, T_n):=\sup\biggl|\sum_{\alpha\in F_n^+}\sum_{j=1}^n\inner{h_{\alpha}}{T_jh_{g_j\alpha}}\biggr|, \]
where $F_n$ is the free group on $n$ generators $g_1,\dots,g_n$, and the supremum is taken over all families of vectors
$\{h_{\alpha}\}_{\alpha\in F_n^+}\subseteq \cl{H}$ with $\sum_{\alpha\in F_n^+}\|h_{\alpha}\|^2=1$.

It is shown in the same paper that $w(T_1,\dots,T_n)=w(S_1\otimes T^*_1+\cdots+S_n\otimes T_n^*)$ \cite[Corollary 1.2]{POP09}, where $w$ on the
right hand side is the numerical radius of an operator on $\cl{H}$ defined in the usual way. Thus, it is natural to extend the notion of  joint numerical radii of $n$-tuples to
the category of $C^*$-algebras.
\begin{defn}
Let $\cl{A}$ be a $C^*$-algebra. The \emph{joint numerical radius} of the $n$-tuple $(a_1,\dots,a_n)\in \cl{A}$ is:
\[ w(a_1,\dots,a_n):=w(S_1\otimes a_1^*+\cdots+S_n\otimes a_n^*),\]
where $S_i$'s are Cuntz isometries.
\end{defn}

\begin{remark}\label{strongrowcontractionnumericalradius}
Let $\cl{A}$ be a $C^*$-algebra. Then the $n$-tuple $(a_1,\dots,a_n)\in\cl{A}$ is a dual row contraction if and only if
\[ w(a_1,\dots,a_n)\leq \frac{1}{2}.\]
\end{remark}

\begin{thm}
Let $E_n'=\Span\{I_{n+1},E_{i0},E_{0i}:1\leq i\leq n\}\subseteq M_{n+1}$, then $\cl{S}_n^d=E_n'$ via the map $\theta:\cl{S}_n^d\to E_n'$, with
$\theta(\delta_0)=I_{n+1}$, $\theta(\delta_i)=E_{0i}$, $\theta(\delta_i^*)=E_{i0}$, for $1\leq i\leq n$.
\end{thm}
\begin{proof}
We first show that $\theta$ is completely positive.
By Corollary \ref{cpstrongrowcontraction} and Remark \ref{strongrowcontractionnumericalradius}, we just need to show that for $n$ Cuntz isometries $S_1,\dots,S_n$,
\[ w\biggl(\begin{pmatrix}
0&S_1&\cdots&S_n\\
0&0&&\\
\vdots&&\ddots&\\
0&&&0
\end{pmatrix}\biggr)\leq \frac{1}{2}  ,\]
which is equivalent to
\[ \begin{pmatrix}
I&zS_1&\cdots&zS_n\\
\bar{z}S_1^*&I&&\\
\vdots&&\ddots&\\
\bar{z}S_n^*&&&I
\end{pmatrix}\geq 0\quad \text{ for all }z\in \bb{T}\]
which clearly holds since $(zS_1,\dots,zS_n)$ is row contraction.

Next, we show that $\theta^{-1}$ is also completely positive. Let $p\in \bb{N}$ and note that $M_p(E_n')=E_n'(M_p)$, we can write a positive element $A\in M_p(E_n')$ as
\[  \begin{pmatrix}
A_0&A_1&\cdots&A_n\\
A_1^*&A_0&&\\
\vdots&&\ddots&\\
A_n^*&&&A_0
\end{pmatrix}, \]
where $A_i\in M_p$.
Consider $\epsilon I_p\otimes I_n+A$, where $I_p$ denotes the identity matrix in $M_p$, for $\epsilon>0$, and let $B=(\epsilon I_p+A_0)^{-\frac{1}{2}}$, we have that
\[ \begin{pmatrix}
I_p&BA_1B&\cdots&BA_nB\\
BA_1^*B&I_p&&\\
\vdots&&\ddots&\\
BA_n^*B&&&I_p
\end{pmatrix}\geq 0 . \]
This implies that $(BA_1B,\dots, BA_nB)$ is a row contraction, and hence
\[ (\theta^{-1})^{(p)}\biggl(\begin{pmatrix}
I_p&BA_1B&\cdots&BA_nB\\
BA_1^*B&I_p&&\\
\vdots&&\ddots&\\
BA_n^*B&&&I_p
\end{pmatrix}\biggr)=I_p\otimes \delta_0+\sum_{i=1}^nBA_iB\otimes \delta_i+\sum_{i=1}^nBA^*_iB\otimes \delta^*_i\geq 0, \]
by Proposition \ref{positiveindual}. Thus,
\[ (\epsilon I_p+A_0)\otimes \delta_0+\sum_{i=1}^nA_i\otimes \delta_i+\sum_{i=1}^nA^*_i\otimes \delta^*_i\geq 0, \quad \text { for all }\epsilon >0 , \]
which  is the same as $(\theta^{-1})^{(p)}(\epsilon I_p\otimes I_n+A)\geq 0$, for all $\epsilon >0$. Since $\theta$ is unital, we know that
$(\theta^{-1})^{(p)}(A)\geq 0$. Hence, $\theta^{-1}$ is also completely positive.
\end{proof}
%\begin{prop}
 %The map $\psi:\cl{S}_n\to \cl{S}_n^d$ defined by is completely positive.
%\end{prop}

We recall the following result of Kavruk:

\begin{thm}\cite{KAV14b}\label{nucleartodual}
Let $\cl{S}$ be a finite dimensional. Then $\cl{S}$ is $C^*$-nuclear if and only if $\cl{S}^d$ is.
\end{thm}

\begin{cor} $E_n'$ is a C*-nuclear operator system.
\end{cor}
\begin{proof} Since $\cl S_n$ is C*-nuclear, $\cl S_n^d$ is C*-nuclear by the above theorem. But $E_n'= \cl S_n^d$ up to complete order isomorphism.
\end{proof}
\begin{remark}
The operator system $E_n'$ seems more elementary to deal with and if we could show directly that $E_n'$ is C*-nuclear, then that would imply by Kavruk's result that $\cl S_n$ is C*-nuclear, which in turn would give another proof of the nuclearity of the Cuntz algebras. However, we have been unable to prove directly that $E_n'$ is $C^*$-nuclear. 
\end{remark}
\section{A Lifting Theorem For Joint Numerical Radius}
The local lifting property of an operator system $\cl{S}$ is defined in \cite{KPTT13}:
\begin{defn}
Let $\cl{S}$ be an operator system, $\cl{A}$ be a unital $C^*$-algebra, $I\lhd\cl{A}$ be an ideal, $q:\cl{A}\to\cl{A}/I$ be the quotient
map and $\phi:\cl{S}\to \cl{A}/I$ be a unital completely positive map.
We say $\phi$ \emph{lifts locally}, if for every finite dimensional
operator system $\cl{S}_0\subseteq \cl{S}$, there exists a completely positive map $\psi:\cl{S}_0\to \cl{A}$ such that $q\circ \psi=\phi$.
We say that $\cl{S}$ has the \emph{operator system locally lifting property} (OSLLP) if for every $C^*$-algebra $\cl{A}$ and every ideal
$I\subseteq \cl{A}$, every unital completely positive map $\phi:\cl{S}\to \cl{A}/I$ lifts locally.
\end{defn}

\begin{thm}\cite{KPTT13}
Let $\cl{S}$ be an operator system, then the following are equivalent:
\begin{enumerate}
\item $\cl{S}$ has the OSLLP;
\item $\cl{S}\otimes_{\min}B(\cl{H})=\cl{S}\otimes_{\max}B(\cl{H})$.
\end{enumerate}
\end{thm}

We have seen that the operator system $\cl{S}_n^d$ is $C^*$-nuclear (Theorem \ref{nucleartodual}). In particular, we have that for a Hilbert space $\cl{H}$,
\[ \cl{S}_n^d\otimes_{\min} B(\cl{H})=\cl{S}_n^d\otimes_{\max}B(\cl{H}) .\]
Thus, the operator system $\cl{S}^d_n$ has the lifting property (LP). \\
By using the LP of $\cl{S}^d_n$, we are able to derive the following result concerning the joint numerical radius.
\begin{thm}
 Let $\cl{A}$ be a unital $C^*$-algebra and $J\lhd \cl{A}$ be an ideal.
 Suppose $T_1+J,\dots,T_n+J\in \cl{A}/J$, then there exist $W_1,\dots,W_n\in \cl{A}$ with $W_i+J=T_i+J$ for each
 $1\leq i\leq n$, such that $w(W_1,\dots,W_n)=w(T_1+J,\dots,T_n+J)$.
\end{thm}
\begin{proof}
Suppose $w(T_1+J,\dots,T_n+J)=K$. If $K=0$, then clearly $T_i+J=0$ for each $1\leq i\leq n$. So we can choose $W_i=0$ for every $1\leq i\leq n$.

So we consider the case when $K>0$. A little scaling shows  that
\[ w(\frac{T_1}{2K}+J,\dots,\frac{T_n}{2K}+J)=\frac{1}{2}  .\]
So the linear map $\phi:\cl{S}_n^d\to \cl{A}/J$ defined by
\[ \phi(\delta_0)=I+J, \quad \phi(\delta_i)=\frac{T^*_i}{2K}+J, \quad \phi(\delta_i^*)=\frac{T_i}{2K}+J \]
is unitally completely positive.

By the argument before the theorem, we know that $\cl{S}_n^d$  has the LP, so there exists a unitally completely positive map
$\hat{\phi}:\cl{S}_n^d\to \cl{A}$ such that $\pi\circ\hat{\phi}=\phi$, where $\pi$ denotes the canonical map from $\cl{A}$ onto $\cl{A}/J$.
Let $W^*_i=2K\hat{\phi}(\delta_i)$, we have that $W^*_i+J=T_i+J$. Moreover, by proposition....., we know that $(\frac{W_1}{2K},\dots,\frac{W_n}{2K})$ is a co-row
contraction. Hence, we have that
\[ w(W_1,\dots,W_n)\leq K  .\]
Now, to complete the proof, we need to show that $w(W_1,\dots,W_n)=K$. Suppose that
\[ w(\frac{W_1}{2K},\dots,\frac{W_n}{2K})<\frac{1}{2}.\]
Then there exists an $\epsilon>0$, such that
\[ w(\frac{(1+\epsilon)W_1}{2K},\dots,\frac{(1+\epsilon)W_n}{2K})<\frac{1}{2}.\]
However, this implies that
\[ I\otimes 1+\sum_{i=1}^nS_i\otimes \frac{(1+\epsilon)W^*_i}{2K}+\sum_{i=1}^nS^*_i\otimes \frac{(1+\epsilon)W_i}{2K}\geq 0, \]
in $\cl{S}_n\otimes_{\min}\cl{A}$. Since $\operatorname{id}\otimes \pi$ is completely positive, we further have that
\[ I\otimes 1+J+\sum_{i=1}^nS_i\otimes \frac{(1+\epsilon)T^*_i+J}{2K}+\sum_{i=1}^nS^*_i\otimes \frac{(1+\epsilon)T_i+J}{2K}\geq 0.\]
It now follows that
\[ w(T_1+J,\dots,T_n+J)\leq \frac{K}{1+\epsilon}, \]
which is a contradiction.
\end{proof}

%\section{$\cl{O}_2\otimes_{\min}\cl{O}_2=\cl{O}_2$}
%\section{$M_p\otimes_{\min}\cl{O}_n$}

\iffalse
imply that the identity map from $\cl{S}\otimes_{\min} \cl{A}$ onto $\cl{S}\otimes_{\max} \cl{A}$ is completely positive.

%we can show that $\phi\otimes \operatorname{id}_{\cl{A}}:\cl{E}\otimes_{\min} \cl{A}\to\cl{S}\otimes_{\min} \cl{A}$
%is a complete quotient map. That is,
To show that
$\cl{S}$ is $C^*$-nuclear,  we can show instead that every positive element in $M_p(\cl{S}\otimes_{\min} \cl{A})^+$ can be lifted to
an element in  $M_p(\cl{E}\otimes_{\min} \cl{A})$ which is the sum of an element in $M_p(\cl{E}\otimes_{\min} \cl{A})^+$ and a element in
$M_p(\Ker \phi\otimes \operatorname{id}_{\cl{A}})$. That is,
\fi

%\bibliographystyle{alpha}

%\bibliography{operatorsystem}

\end{document}